\documentclass[a4paper,11pt,leqno]{amsart}
\usepackage{amsfonts}
\usepackage{amssymb}
\usepackage{amsmath}
\usepackage{graphicx}
\usepackage{wrapfig}

\newcommand{\N}{\mathbb N}

\newtheorem{theorem}{Theorem}[section]
\newtheorem{cor}{Corollary}[section]

\newtheorem{lemma}{Lemma}[section]

\newtheorem{rem}{Remark}[section]

\newcommand{\Z}{\mathbb Z}

\usepackage{hyperref}
\usepackage{float}
\usepackage{enumitem}
\usepackage{soul}
\usepackage[normalem]{ulem}

\usepackage{pgf,tikz,pgfplots}
\pgfplotsset{compat=1.13}
\definecolor{qqqqff}{rgb}{0,0,1}

\theoremstyle{plain}

\theoremstyle{remark}

\newtheorem*{notation}{Notation}

\makeatletter
\@addtoreset{footnote}{page}
\makeatother

\definecolor{ttzzqq}{rgb}{0.2,0.6,0}

\begin{document}

\title{}

\markboth{Christian Aebi and Grant Cairns}  {Lattice Equable Quadrilaterals II}

\begin{minipage}[0.8 \textheight]{0.8 \textwidth}
 \begin{flushleft}
 \end{flushleft}
\end{minipage}

\centerline{}\bigskip

\centerline {\Large{\bf LATTICE EQUABLE QUADRILATERALS II }}\centerline{}%
\centerline {\Large{\bf - KITES, TRAPEZOIDS  \& CYCLIC QUADRILATERALS}}

\bigskip

\begin{center}
{\large CHRISTIAN AEBI  AND   GRANT CAIRNS}

\centerline{}
\end{center}

\bigskip

\textbf{Abstract.} We show that there are 4 infinite families of lattice equable kites, given by corresponding Pell or Pell-like equations, but up to Euclidean motions, there are exactly 5 lattice equable trapezoids (2 isosceles, 2 right, 1 singular) and 4 lattice equable cyclic quadrilaterals. We also show that, with one exception, the interior diagonals of lattice equable quadrilaterals are irrational.

\bigskip

\section{Introduction} 

A \emph{lattice equable quadrilateral} (LEQ for short) is a  quadrilateral whose vertices lie on the integer lattice $\Z^2$ and which is equable in the sense that its area equals its perimeter.
The general context here is the field of study related to Heronian triangles and their generalisations. Recall that a \emph{Heronian triangle}, or \emph{Heron triangle} is a triangle with integer side lengths and integer area \cite{ch,re}. Similarly, quadrilaterals with integer side lengths, integer diagonals and integer area have been studied for some time; for the rational version, see \cite{Ku} and \cite[pp.~216--221]{Di}.  \emph{Brahmagupta quadrilaterals}, which are cyclic quadrilaterals with integer side lengths, integer diagonals and integer area, have received particular attention \cite{sa,NS}. The objects of our study are somewhat different. While LEQs have integer side lengths (see \cite[Remark 2]{AC}), they do not necessarily have diagonals of integer lengths (see Theorem~\ref{T:irrat} below). On the other hand, the requirement that the vertices be lattice points is often a significant restriction.

This paper is a continuation of the work \cite{AC} which treated lattice equable parallelograms. It is shown in  \cite{AC} that lattice equable parallelograms are quite abundant, and  naturally form a forest of three infinite trees all of whose vertices  have degree 4, bar the roots. This present paper studies LEQs with other specials properties: kites, trapezoids and cyclic quadrilaterals.

Recall that a  \emph{kite} is a planar quadrilateral having two disjoint pairs of adjacent sides of equal length; see \cite[Chap.~7]{UG}.
A \emph{dart} is a concave kite.
If a lattice equable kite $Q$ has vertices $O(0,0)$, $A(x,y)$, $B(z,w)$, $C(u,v)$ and is symmetrical about the diagonal $OB$, then we can 
cut $Q$ along $OB$ and rearrange to form the parallelogram $O(0,0),A(x,y),B(z,w),C'(z-x,w-y)$, which by construction is a lattice equable parallelogram; see Figure~\ref{F:KP}. 
By the reverse procedure, a lattice equable parallelogram $OABC'$ produces a lattice equable kite  $OABC$ provided the reflection $C$ of $A$ in the diagonal $OB$ is a lattice point.
Thus lattice equable kites correspond to a particular subfamily of lattice equable parallelograms. Not all lattice equable parallelograms produce lattice equable kites; for example, as shown in Figure \ref{F:36}, the $3\times 6$ rectangle does not produce a lattice equable kite. Here the point $C=(-\frac95,\frac{12}5)$ is not a lattice point.  In fact, there is no  lattice equable kite with side lengths $3$ and $6$, as we will see from the corollary of the following theorem.

\begin{theorem}\label{Th1}
The four families \emph{K1 -- K4} given in the following table provide the list, without redundancy, of all  lattice equable kites, up to Euclidean motions. In each case, the points $O(0,0),A,B,C$ are the vertices of a lattice equable kite with axis of symmetry $OB$, and $M$ is the midpoint of the segment $AC$.
In each family, the nonnegative integers $n,i$ are solutions to a given Pell or Pell-like equation. In family K2 we require $n>1$.  

\smallskip
\noindent
\resizebox{\columnwidth}{!}{%
\begin{tabular}{c|c|c|c|c|c}
  \hline
   Case & Equation & $M$ &$A$ & $B$  & $C$ \\\hline
  \emph{K1}& $n^2-5i^2=4$ & $\frac12(n+5i)(2,1)$& $M+(2,-4)$  & $n(2,1)$  & $M-(2,-4)$  \\
  \emph{K2} & $n^2-5i^2=1$& $ (2n+5i)(2,1)$ &$M+(1,-2)$ & $4n(2,1)$  & $M-(1,-2)$ \\
\emph{K3} & $n^2-2i^2=1$& $(n+2i)(2,2) $ & $M+(2,-2)$& $ 4n(1,1)$ & $M-(2,-2)$\\
 \emph{K4} & $2n^2-i^2=1$& $ (4n+3i)(\frac32,\frac32)$ & $M+(\frac32,-\frac32)$ &$12n(1,1)$ & $M-(\frac32,-\frac32)$ \\
\hline
\end{tabular}
}
\end{theorem}

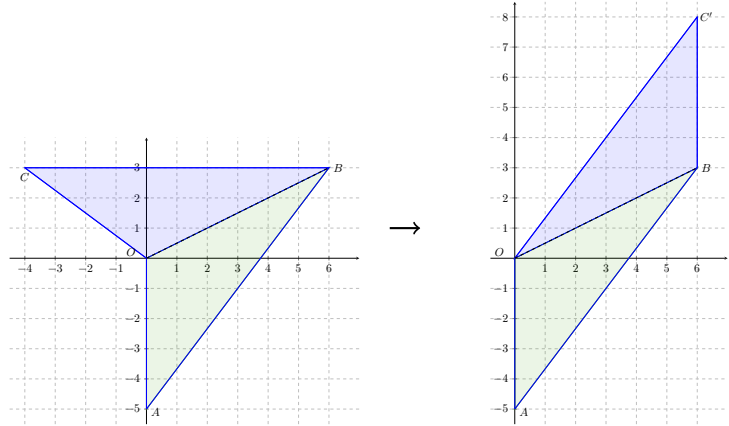
\begin{figure}
\begin{tikzpicture}[scale=.4][line cap=round,line join=round,>=triangle 45,x=1cm,y=1cm]
\begin{axis}[
x=1cm,y=1cm,
axis lines=middle,
grid style=dashed,
ymajorgrids=true,
xmajorgrids=true,
xmin=-4.5,
xmax=7,
ymin=-5.5,
ymax=4,
xtick={-4,-3,...,6},
ytick={-5,-4,...,3},]
\draw[color=ttzzqq,fill=ttzzqq,fill opacity=0.1] (0,0) --  (6,3) -- (0,-5)  -- cycle; 
\draw[color=qqqqff,fill=qqqqff,fill opacity=0.1] (0,0) --  (-4,3) -- (6,3)   -- cycle; 
\draw [line width=1pt,color=qqqqff] (6,3)-- (-4,3);
\draw [line width=1pt,color=qqqqff] (6,3)-- (0,-5);
\draw [line width=1pt,color=qqqqff] (0,0)-- (-4,3);
\draw [line width=1pt,color=qqqqff] (0,0)-- (0,-5);
\draw [line width=1pt,dashed,color=black] (0,0)-- (6,3);
\draw [color=black] (-0.5,0.2) node {$O$};
\draw [color=black] (0.3,-5.1) node {$A$};
\draw [color=black] (6.3,3) node {$B$};
\draw [color=black] (-4,2.7) node {$C$};
\end{axis}
\draw [->,line width=0.8pt] (12.5,6.5) -- (13.5,6.5);
\end{tikzpicture}
\qquad
\begin{tikzpicture}[scale=.4][line cap=round,line join=round,>=triangle 45,x=1cm,y=1cm]
\begin{axis}[
x=1cm,y=1cm,
axis lines=middle,
grid style=dashed,
ymajorgrids=true,
xmajorgrids=true,
xmin=-.8,
xmax=7,
ymin=-5.5,
ymax=8.5,
xtick={0,1,...,6},
ytick={-5,-4,...,8},]
\draw[color=ttzzqq,fill=ttzzqq,fill opacity=0.1] (0,0) -- (0,-5) -- (6,3)  -- cycle; 
\draw[color=qqqqff,fill=qqqqff,fill opacity=0.1] (0,0) --  (6,3) --  (6,8) -- cycle; 
\draw [line width=1pt,color=qqqqff] (6,3)-- (6,8);
\draw [line width=1pt,color=qqqqff] (6,3)-- (0,-5);
\draw [line width=1pt,color=qqqqff] (0,0)-- (6,8);
\draw [line width=1pt,color=qqqqff] (0,0)-- (0,-5);
\draw [line width=1pt,dashed,color=black] (0,0)-- (6,3);
\draw [color=black] (-0.5,0.2) node {$O$};
\draw [color=black] (0.3,-5.1) node {$A$};
\draw [color=black] (6.3,3) node {$B$};
\draw [color=black] (6.3,8) node {$C'$};
\end{axis}
\end{tikzpicture}
\caption{Lattice equable kites produce lattice equable parallelograms}\label{F:KP}
\end{figure}

\begin{figure}
\begin{tikzpicture}[scale=.5][line cap=round,line join=round,>=triangle 45,x=1cm,y=1cm]
\begin{axis}[
x=1cm,y=1cm,
axis lines=middle,
grid style=dashed,
ymajorgrids=true,
xmajorgrids=true,
ymin=-.5,
ymax=7,
xmin=-.5,
xmax=4,
ytick={0,1,...,6},
xtick={0,1,...,4},]
\draw[color=ttzzqq,fill=ttzzqq,fill opacity=0.1] (0,0) -- (0,6) -- (3,6) -- (3,0) -- cycle; 
\draw[line width=1pt,color=qqqqff] (0,0) -- (0,6) -- (3,6) -- (3,0) -- cycle; 
\draw [line width=1pt,dotted,color=black] (0,0)-- (3,6);
\draw [color=black] (-0.2,-0.2) node {$O$};
\draw [color=black] (3.2,.2) node {$A$};
\draw [color=black] (3.2,6.2) node {$B$};
\draw [color=black] (0.25,6.25) node {$C'$};
\end{axis}
\draw [->,line width=0.8pt] (5.5,3.5) -- (6.5,3.5);
\end{tikzpicture}
\quad 
\begin{tikzpicture}[scale=.5][line cap=round,line join=round,>=triangle 45,x=1cm,y=1cm]
\begin{axis}[
x=1cm,y=1cm,
axis lines=middle,
grid style=dashed,
ymajorgrids=true,
xmajorgrids=true,
ymin=-.5,
ymax=6.5,
xmin=-2.5,
xmax=5,
ytick={0,1,...,6},
xtick={-2,-1,...,5},]
\draw[color=ttzzqq,fill=ttzzqq,fill opacity=0.1] (0,0) --  (-9/5,12/5) --(3,6) -- (3,0) -- cycle; 
\draw[line width=1pt,color=qqqqff] (0,0) --  (-9/5,12/5) --(3,6) -- (3,0) -- cycle; 
\draw [line width=1pt,dotted,color=black] (0,0)-- (3,6);
\draw [color=black] (-0.2,-0.2) node {$O$};
\draw [color=black] (3.2,.2) node {$A$};
\draw [color=black] (3.2,6.2) node {$B$};
\draw [color=black] (-2,2) node {$C$};
\draw[line width=.5pt] (1.5,3) circle (3.354);
\end{axis}
\end{tikzpicture}
\caption{A lattice equable parallelogram that fails to produce a lattice equable kite}\label{F:36}
\end{figure}
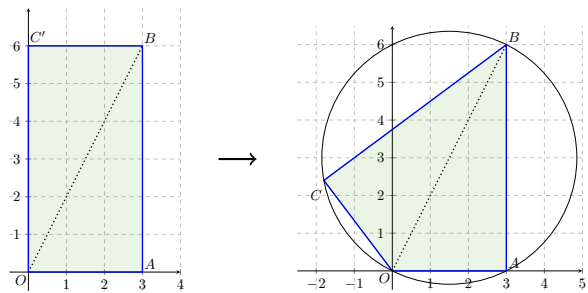

\begin{cor}\label{co}
Up to Euclidean motions, there are only three convex lattice equable kites. They are:
\begin{enumerate}
\item[(a)]  The rhombus of side length 5 of Figure~\ref{F:rhom}.
\item[(b)]  The $4\times 4$ square.
\item[(c)]  The lattice equable kite with side lengths 3 and 15 shown in Figure~\ref{F:315}.
\end{enumerate}
\end{cor}

The proof of this theorem and its corollary are given in the next Section; it replies heavily on the results of \cite{AC}. 
Section~\ref{S:rm} gives more details about the families K1 -- K4 of Theorem~\ref{Th1}.


We now turn to equable trapezoids. As discussed in \cite{UG}, the term `trapezoid' is not universal. For us, a \emph{trapezoid} is a quadrilateral that has exactly one pair of parallel sides. 
An \emph{isosceles trapezoid} is a trapezoid for which the nonparallel sides have equal length. Thus an isosceles trapezoid is the truncation of an 
isosceles triangle by a line parallel to the base.
A \emph{right trapezoid} is a trapezoid having two adjacent right angles.

In \cite{do}, Dolan briefly outlines how equable trapezoids might be classified by reducing the problem to the study of certain triangles; \cite{do} concludes with the words: ``A nice exercise for an interested reader is then to determine precisely which triangles produce equable trapezia in this fashion''. Indeed, this nice exercise gives the following result.

\begin{theorem}\label{T:traps}
Up to Euclidean motions, there are only five equable trapezoids having integer sides and area. They are the two right trapezoids shown in Figure \ref{F:lert}, the two isosceles trapezoids shown in Figure~\ref{F:it510}, and the trapezoid of side lengths 20,4,15,3, shown in Figure~\ref{F:friend}.
\end{theorem}

The proof of this theorem is given in the Section \ref{S:prooftraps}.

\begin{rem} The five equable trapezoids of Theorem~\ref{T:traps} are
all LEQs. None of them are obtuse. The trapezoid of side lengths 20,4,15,3 is particularly interesting. First, it has 
two orthogonal non-adjacent edges. Secondly, it cannot be drawn as a LEQ with horizontal parallel sides.
\end{rem}

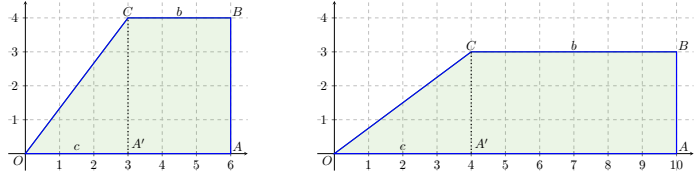
\begin{figure}
\begin{tikzpicture}[scale=.45][line cap=round,line join=round,>=triangle 45,x=1cm,y=1cm]
\begin{axis}[
x=1cm,y=1cm,
axis lines=middle,
grid style=dashed,
ymajorgrids=true,
xmajorgrids=true,
xmin=-.5,
xmax=6.5,
ymin=-.5,
ymax=4.5,
xtick={0,1,...,6},
ytick={0,1,...,5},]
\draw[color=ttzzqq,fill=ttzzqq,fill opacity=0.1] (0,0) -- (6,0) -- (6,4) -- (3,4) -- cycle; 
\draw[line width=1pt,color=qqqqff] (0,0) -- (6,0) -- (6,4) -- (3,4) -- cycle; 
\draw [line width=1pt,dotted,color=black] (3,0)-- (3,4);
\draw [color=black] (-0.2,-0.2) node {$O$};
\draw [color=black] (6.2,.2) node {$A$};
\draw [color=black] (3.3,.3) node {$A'$};
\draw [color=black] (6.2,4.2) node {$B$};
\draw [color=black] (3,4.2) node {$C$};
\draw [color=black] (1.5,0.2) node {$c$};
\draw [color=black] (4.5,4.2) node {$b$};
\end{axis}
\quad 
\end{tikzpicture}
\qquad\begin{tikzpicture}[scale=.45][line cap=round,line join=round,>=triangle 45,x=1cm,y=1cm]
\begin{axis}[
x=1cm,y=1cm,
axis lines=middle,
grid style=dashed,
ymajorgrids=true,
xmajorgrids=true,
xmin=-.5,
xmax=10.5,
ymin=-.5,
ymax=4.5,
xtick={0,1,...,10},
ytick={0,1,...,5},]
\draw[color=ttzzqq,fill=ttzzqq,fill opacity=0.1] (0,0) -- (10,0) -- (10,3) -- (4,3) -- cycle; 
\draw[line width=1pt,color=qqqqff] (0,0) -- (10,0) -- (10,3) -- (4,3) -- cycle; 
\draw [line width=1pt,dotted,color=black] (4,0)-- (4,3);
\draw [color=black] (-0.2,-0.2) node {$O$};
\draw [color=black] (10.2,.2) node {$A$};
\draw [color=black] (4.3,.3) node {$A'$};
\draw [color=black] (10.2,3.2) node {$B$};
\draw [color=black] (4,3.2) node {$C$};
\draw [color=black] (2,0.2) node {$c$};
\draw [color=black] (7,3.2) node {$b$};
\end{axis}
\end{tikzpicture}
\caption{The 6,4,3,5 and 10,3,6,5 lattice equable right trapezoids}\label{F:lert}
\end{figure}

\begin{figure}
\begin{tikzpicture}[scale=.45][line cap=round,line join=round,>=triangle 45,x=1cm,y=1cm]
\begin{axis}[
x=1cm,y=1cm,
axis lines=middle,
grid style=dashed,
ymajorgrids=true,
xmajorgrids=true,
xmin=-.5,
xmax=9,
ymin=-.5,
ymax=5,
xtick={0,1,...,8},
ytick={0,1,...,4},]
\draw[color=ttzzqq,fill=ttzzqq,fill opacity=0.1] (0,0) -- (8,0) -- (5,4) -- (3,4) -- cycle; 
\draw [line width=1pt,color=qqqqff] (5,4)-- (3,4);
\draw [line width=1pt,color=qqqqff] (3,4)-- (0,0);
\draw [line width=1pt,color=qqqqff] (0,0)-- (8,0);
\draw [line width=1pt,color=qqqqff] (8,0)-- (5,4);
\draw [color=black] (-0.2,-0.2) node {$O$};
\draw [color=black] (8.2,.3) node {$A$};
\draw [color=black] (5.2,4.3) node {$B$};
\draw [color=black] (3,4.3) node {$C$};
\draw [color=black] (6.2,.3) node {$A'$};
\draw [line width=1pt,dashed,color=black] (6,0)-- (3,4);
\end{axis}
\end{tikzpicture}
\begin{tikzpicture}[scale=.45][line cap=round,line join=round,>=triangle 45,x=1cm,y=1cm]
\begin{axis}[
x=1cm,y=1cm,
axis lines=middle,
grid style=dashed,
ymajorgrids=true,
xmajorgrids=true,
xmin=-.5,
xmax=15,
ymin=-.5,
ymax=4,
xtick={0,1,...,14},
ytick={0,1,...,3},]
\draw[color=ttzzqq,fill=ttzzqq,fill opacity=0.1] (0,0) -- (14,0) -- (10,3) -- (4,3) -- cycle; 
\draw [line width=1pt,color=qqqqff] (10,3)-- (4,3);
\draw [line width=1pt,color=qqqqff] (4,3)-- (0,0);
\draw [line width=1pt,color=qqqqff] (0,0)-- (14,0);
\draw [line width=1pt,color=qqqqff] (14,0)-- (10,3);
\draw [color=black] (-0.2,-0.2) node {$O$};
\draw [color=black] (14.2,.3) node {$A$};
\draw [color=black] (10.2,3.3) node {$B$};
\draw [color=black] (4,3.3) node {$C$};
\draw [color=black] (8.2,.3) node {$A'$};
\draw [line width=1pt,dashed,color=black] (8,0)-- (4,3);
\end{axis}
\end{tikzpicture}
\caption{The 8,5,2,5 and 14,5,6,5 lattice equable isosceles trapezoids}\label{F:it510}
\end{figure}
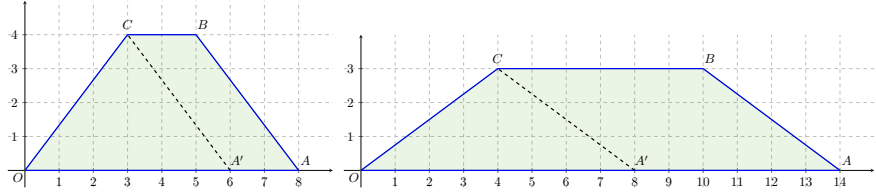

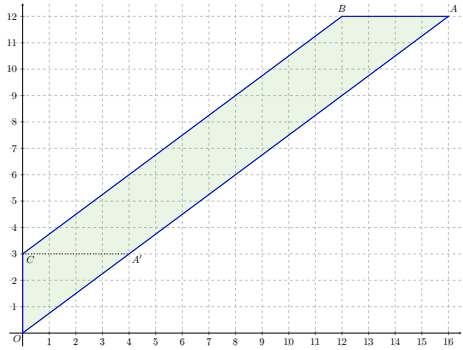
\begin{figure}
\begin{tikzpicture}[scale=.35][line cap=round,line join=round,>=triangle 45,x=1cm,y=1cm]
\begin{axis}[
x=1cm,y=1cm,
axis lines=middle,
grid style=dashed,
ymajorgrids=true,
xmajorgrids=true,
xmin=-.5,
xmax=16.5,
ymin=-.5,
ymax=12.5,
xtick={0,1,...,17},
ytick={0,1,...,13},]
\draw[color=ttzzqq,fill=ttzzqq,fill opacity=0.1] (0,0)--(16,12)--(12,12)--(0,3)-- cycle; 
\draw[line width=1pt,color=qqqqff] (0,0)--(16,12)--(12,12)--(0,3) -- cycle; 
\draw [line width=1pt,dotted,color=black] (0,3)-- (4,3);
\draw [color=black] (-0.2,-0.2) node {$O$};
\draw [color=black] (16.2,12.3) node {$A$};
\draw [color=black] (4.3,2.8) node {$A'$};
\draw [color=black] (12,12.3) node {$B$};
\draw [color=black] (.3,2.8) node {$C$};
\end{axis}
\end{tikzpicture}
\caption{The 20,4,15,3 lattice equable trapezoid}\label{F:friend}
\end{figure}


A \emph{cyclic quadrilateral} is a quadrilateral whose vertices lie on a circle. The classification of cyclic LEQs that we give below in Theorem~\ref{Th4} is not really new.  The determination of the possible integer side lengths for equable cyclic quadrilaterals was established by Nelson \cite{Ne}. It  also appears in the web pages of Balmoral Software \cite{BS} where the problem is shown to amount to a finite number of cases  that is large but can be performed by computer. The proof of Theorem~\ref{Th4} that we give in Section~\ref{S:LECQs} is virtually identical to that of \cite{Ne} until near the end where a simplification allows us to reduce the problem to just 63 cases, that the authors have resolved both by computer and by ``bare-hands''. The LEQ hypothesis is then only used to exclude certain cyclic orders of the sides.

\begin{theorem}  \label{Th4}
Up to Euclidean motions, there are only four  cyclic LEQs; they are the $4\times 4$ square, the $3\times 6$ rectangle, and the two  isosceles trapezoids shown in Figure~\ref{F:it510}.
 \end{theorem}

The paper concludes with a general result. It is shown in \cite{AC} that the diagonals of all lattice equable parallelograms are irrational. This is not the case for all LEQs. For example, for the concave  LEQ with vertices 
\[
O(0,0),A(20,15),B(8,10),C(8,15),
\]
 the external diagonal $AC$ has length 12. Nevertheless, one has the following result, which is proved in Section~\ref{S:irratdiag}.

\begin{theorem}\label{T:irrat}
For a lattice equable quadrilateral $OABC$, the interior diagonals are irrational, except if $OABC$ is congruent to the right trapezoid  of side lengths 6,4,3,5 shown on the left of Figure~\ref{F:lert}. 
\end{theorem}

\begin{notation} In this paper, a quadrilateral $OABC$ is defined by four vertices $O,A,B,C$, no three of which are colinear, such that the line segments joining the vertices have no interior points of intersection; that is, our quadrilaterals have no self-intersections. We will always write the vertices $O,A,B,C$ in positive (counterclockwise) cyclic order. We use the notation  $K(OABC)$ for area and  $P(OABC)$ for perimeter. 
For ease of expression, we often simply write $K$ for $K(OABC)$, and  $P$ for $P(OABC)$, and we abbreviate the triangle areas $K(OAB),K(BCO)$ as $K_A,K_C$ respectively, and the perimeter $P(BCO)$ as $P_C$.
By abuse of notation, we write  $OA,AB,BC,CO$  for both the sides, and their lengths; the meaning should be clear from the context. 
We also usually denote $OA,AB,BC,CO$ by the letters $a,b,c,d$. The lengths of the diagonals $OB,AC$ are denoted $p,q$, respectively. 
We use vector notation, such as  $\overrightarrow{AB}$.  But we use the same symbol, $A$ say, for the vertex $A$ and its position vector $\overrightarrow{OA}$. 
In this paper, we employ the term \emph{positive} in the strict sense. So $\N=\{n\in\Z \ |\ n>0\}$.
\end{notation}

\section{Proof of Theorem \ref{Th1} and Corollary \ref{co}} 

\noindent\textbf{Proof of Theorem \ref{Th1}.}
First note that in each family, $C$ is visibly the reflection of $A$ in the line through $OB$, so $OABC$ is a kite. Secondly, the vertices $A,B,C$ are lattice points. Indeed, this is obvious for K2 and K3, while for K1, it follows from the fact that if $n^2-5i^2=4$, then $n$ and $i$ have the same parity. For K4, if $2n^2-i^2=1$, then $i$ is odd and so $M$ is not a lattice point, but $A,B,C$ are. To see that $OABC$ is equable, we need to show that $K_A$  is the sum of the length $a=OA$ and $b=AB$. Note that $K_A=\frac12 OB\cdot  MA$. Routine calculations give the following values:
\begin{center}
\begin{tabular}{c|c|c|c}
  \hline
   Case &  $K_A$ & $a$ & $b$  \\\hline
  K1& $5n$& $\frac52(n+i)$& $\frac52(n-i)$ \\
  K2 & $10n$& $5(n+2i)$& $5(n-2i)$\\
K3 & $8n$& $4(n+i)$& $4(n-i)$\\
 K4 & $18n$& $3(3n+2i)$& $3(3n-2i)$ \\
\hline
\end{tabular}
\end{center}
From the above table, $K_A=a+b$, so the members of the families K1-K4 are all equable.

Conversely, consider a lattice equable kite with vertices
\[
O(0,0),A(x,y),B(z,w),C(u,v),
\]
that is symmetrical about the diagonal $p=OB$. By translation, and reflection in the $x$-axis,  $y$-axis and the lines $y=\pm x$, if necessary, we may assume without loss of generality that the diagonal $p$ lies in the first quadrant and has gradient $\le 1$. Furthermore,  by reflection in the $x$ and $y$ axes and translating, we may assume that $a\ge b$ and that $O,A,B,C$ are in positive cyclic order.

Notice that $O(0,0),A(x,y),B(z,w),C'(z-x,w-y)$ is a lattice equable parallelogram. Suppose initially that the diagonal $p=OB$ is the shortest of the two diagonals.
From \cite[Lemma 2]{AC}, $p^2=a^2+b^2- 2\sqrt{a^2b^2-4(a+b)^2}$ and thus from \cite[Lemma 6b]{AC}, the diagonal  $AC=q$ satisfies
\[
q^2=\frac{16(a+b)^2}{p^2}=\frac{16(a+b)^2}{a^2+b^2- 2\sqrt{a^2b^2-4(a+b)^2}}.
\]
Furthermore, from \cite[Section~3]{AC}, there  are positive integers $m,n$ with $m\le n$ for which we have the Vieta relations
\begin{equation}\label{E:bc}
ab= k(m^2+n^2),\quad a+b=kmn,\quad \sqrt{a^2b^2 -4(a+b)^2}= k(n^2-m^2),
\end{equation}
where $k$ is either 5, 8 or 9.
It follows that, as in the proof of \cite[Theorem~5]{AC}, $p^2=k^2m^2n^2-4kn^2$. Hence
\begin{equation}\label{E:AC}
q^2=\frac{16(a+b)^2}{a^2+b^2- 2\sqrt{a^2b^2-4(a+b)^2}}=\frac{16k^2m^2n^2}{k^2m^2n^2-4kn^2}=\frac{16km^2}{km^2-4}.
\end{equation}
Since $q^2$ is the square of the distance between two lattice points, it is a positive integer, though $q$ may well be irrational.
Thus $16km^2=q^2(km^2-4)$, so
 $q^2\ge 16$ and 
\[
m^2=\frac{4q^2}{k(q^2-16)}.
\]
As $m\ge 1$, we have $q^2\le \frac{16k}{k-4}$. Thus, as $k\in\{5,8,9\}$, we have three cases:
\begin{enumerate}
\item[(a)]  For $k=5$, we have $m^2=\frac{4q^2}{5(q^2-16)}$ and $q^2\le 80$. Calculations show that the only possibilities are (a1) $m=1,q^2=80$ and (a2) $m=2,q^2=20$.

\item[(b)]  For $k=8$, we have $m^2=\frac{q^2}{2(q^2-16)}$ and $q^2\le 32$. Calculations give just the one possibility $m=1,q^2=32$.

\item[(c)]  For $k=9$, we have $m^2=\frac{4q^2}{9(q^2-16)}$ and $q^2\le \frac{144}5$ so $q^2\le 29$. Calculations give just the one possibility $m=2,q^2=18$.
\end{enumerate}

We now examine these four possibilities in greater detail:

{\bf (a1).} For $k=5,m=1,q^2=80$: one has $a=\frac12(5n+\sqrt{5n^2-20})$ from \eqref{E:bc}. So we require that $5n^2-20$ be a square, say $5n^2-20=j^2$. Note that $j$ must be a multiple of 5, say $j=5i$. So
$n^2-4=5i^2$. Thus $a=\frac12(5n+5i)$.
In this case $b=\frac12(5n-5i)$, again from \eqref{E:bc}
Note that $q^2=80$ is the sum of two squares in only one possible way: $80=8^2+4^2$. So,  as  $O,A,B,C$ are in positive cyclic order, the vector $\overrightarrow{AC}$ is necessarily $(-4,8)$.
From above,
$p^2=k^2m^2n^2-4kn^2=5n^2$.
Since $OB$ is orthogonal to $AC$, we have $B=n(2,1)$. Since $a$ and $b$ are fixed, the triangle $OAB$ is thus uniquely determined. The vertices $A,B,C$ are therefore those of family K1.

{\bf (a2).} For $k=5,m=2,q^2=20$: one has $a=5n+2\sqrt{5n^2-5}$. So we require that $5n^2-5$ be a square, say $5n^2-5=j^2$. Once again, $j$ must be a multiple of 5, say $j=5i$. So
$n^2-1=5i^2$. Thus $a=5n+10i$ and $b=5n-10i$.
Note that $q^2=20$ is the sum of two squares in only one possible way:  $20=4^2+2^2$.
So the vector $\overrightarrow{AC}$ is necessarily $(-2,4)$.
The diagonal $p$ satisfies 
$p^2=k^2m^2n^2-4kn^2=80n^2$.
Since $OB$ is orthogonal to $AC$, we have $B=4n(2,1)$. Since $a$ and $b$ are fixed, the triangle $OAB$ is thus uniquely determined. The vertices $A,B,C$ are therefore those of family K2.

{\bf (b).} For $k=8,m=1$: one has $a=2(2n+\sqrt{2n^2-2})$. So we require that $2n^2-2$ be a square, say $2n^2-2=j^2$. Note that $j$ must be even, say $j=2i$. So
$n^2-1=2i^2$. Thus $a=4(n+i)$ and $b=4(n-i)$.
Note that $q^2=32$ is the sum of two squares in only one possible way:  $32=4^2+4^2$.
So the vector $\overrightarrow{AC}$ is necessarily $(-4,4)$.
The diagonal $p$ satisfies 
$p^2=k^2m^2n^2-4kn^2=32n^2$.
Since $OB$ is orthogonal to $AC$, we have $B=4n(1,1)$. Since $a$ and $b$ are fixed, the triangle $OAB$ is thus uniquely determined. The vertices $A,B,C$ are therefore those of family K3.

{\bf (c).} For $k=9,m=2$: one has $a=3(3n+2\sqrt{2n^2-1})$. So we require that $2n^2-1$ be a square, say $2n^2-1=i^2$. So $a=3(3n+2i)$ and $b=3(3n-2i)$.
Note that $q^2=18$ is the sum of two squares  in only one possible way:  $18=3^2+3^2$.
So the vector $\overrightarrow{AC}$ is necessarily $(-3,3)$.
The diagonal $p$ satisfies 
$p^2=k^2m^2n^2-4kn^2=288n^2$.
Since $OB$ is orthogonal to $AC$, we have $B=12n(1,1)$. Since $a$ and $b$ are fixed, the triangle $OAB$ is thus uniquely determined. The vertices $A,B,C$ are therefore those of family K4 with $n\ge 2$. Notice that there is a member of K4 with $n=1$ which is not included here as $n\ge m=2$. We will find the missing family member below.

\smallskip
The above calculations assumed that $p$ is the shortest diagonal of the parallelogram $OABC'$. Similarly, if $p$ is the longest diagonal, then $n^2=\frac{4q^2}{k(q^2-16)}$, and essentially the same calculations as before  gives the possibilities:
\begin{enumerate}
\item[(i)]  $k=5$; $n=1,q^2=80$ or $n=2,q^2=20$.

\item[(ii)]  $k=8$; $n=1,q^2=32$.

\item[(iii)]   $k=9$; $n=2,q^2=18$.
\end{enumerate}
As $m\le n$, this gives six potential cases:
\begin{enumerate}
\item[(i1)] $k=5$; $m=n=1$. But then \eqref{E:bc}  has no integer solutions for $a,b$.
\item[(i2)] $k=5$; $m=1,n=2$. Then $a=5,b=5$. 
\item[(i3)] $k=5$; $m=n=2$. But then \eqref{E:bc} has no integer solutions for $a,b$.

\item[(ii)] $k=8$; $m=n=1$. Then $a=4,b=4$. 

\item[(iii1)]  $k=9$; $m=1,n=2$. Then $a=15,b=3$.  
\item[(iii2)]  $k=9$; $m=n=2$. But then \eqref{E:bc} has no integer solutions for $a,b$. 
\end{enumerate}
So it remains to consider the  (i2), (ii), (iii1). In Case (i2), $k=5,m=1,n=2, a=b=5$ and long diagonal $p$ given by $p^2=k^2m^2n^2-4km^2=80$. 
But $80$ is the sum of two squares  in only one possible way:  $80=4^2+8^2$.
So $B$ is necessarily $(8,4)$. Then $A$ is forced to be $(5,0)$ and $C$ is $(3,4)$, as shown on the left of Figure~\ref{F:rhom}. The rhombus of side length 5  also appeared as member $n=2,i=0$ of K1; this is shown on the right of Figure~\ref{F:rhom}.  Obviously, in Figure~\ref{F:rhom}, the figure on the left  can be moved onto the figure on the right by a Euclidean motion.

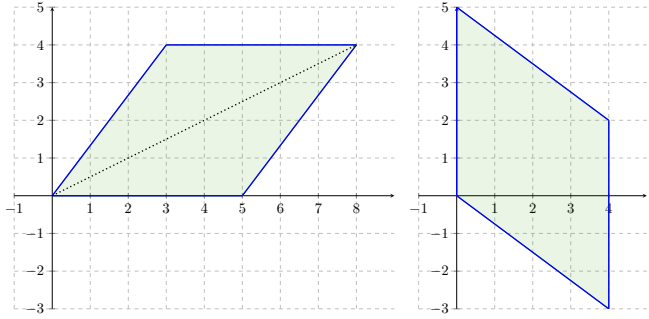
\begin{figure}
\begin{tikzpicture}[scale=.5][line cap=round,line join=round,>=triangle 45,x=1cm,y=1cm]
\begin{axis}[
x=1cm,y=1cm,
axis lines=middle,
grid style=dashed,
ymajorgrids=true,
xmajorgrids=true,
xmin=-1,
xmax=9,
ymin=-3,
ymax=5,
xtick={-1,0,...,8},
ytick={-3,-2,...,5},]
\draw[color=ttzzqq,fill=ttzzqq,fill opacity=0.1] (0,0) -- (5,0) -- (8,4) -- (3,4) -- cycle; 
\draw [line width=1pt,color=qqqqff] (8,4)-- (3,4);
\draw [line width=1pt,color=qqqqff] (3,4)-- (0,0);
\draw [line width=1pt,color=qqqqff] (0,0)-- (5,0);
\draw [line width=1pt,color=qqqqff] (5,0)-- (8,4);
\draw [line width=1pt,dotted,color=black] (0,0)-- (8,4);
\end{axis}
\quad 
\end{tikzpicture}
\begin{tikzpicture}[scale=.5][line cap=round,line join=round,>=triangle 45,x=1cm,y=1cm]
\begin{axis}[
x=1cm,y=1cm,
axis lines=middle,
grid style=dashed,
ymajorgrids=true,
xmajorgrids=true,
xmin=-1,
xmax=5,
ymin=-3,
ymax=5,
xtick={-1,0,...,4},
ytick={-3,-2,...,5},]
\draw[color=ttzzqq,fill=ttzzqq,fill opacity=0.1] (0,0) -- (4,-3) -- (4,2) -- (0,5) -- cycle; 
\draw [line width=1pt,color=qqqqff] (0,0)-- (4,-3);
\draw [line width=1pt,color=qqqqff] (4,2)-- (4,-3);
\draw [line width=1pt,color=qqqqff] (0,0)-- (0,5);
\draw [line width=1pt,color=qqqqff] (0,5)-- (4,2);
\end{axis}
\end{tikzpicture}
\caption{Equable Rhombi}\label{F:rhom}
\end{figure}

Case (ii) is the $4\times 4$ square, so its short and long diagonals are equal: this appeared as member $n=1,i=0$ of K3.

The kite of Case (iii1) has $k=9,m=1,n=2, a=15,b=3$ and long diagonal $p$ given by $p^2=k^2m^2n^2-4km^2=288$. 
But $288$ is the sum of two squares  in only one possible way:  $288= 12^2+12^2$.
So $B$ is necessarily $(12,12)$. Then $A$ is forced to be $(12,9)$ and $C$ is $(9,12)$, as shown in Figure~\ref{F:315}. This kite is the first member of the family K4, with $n=1,i=1$. We have thus shown that the table in the statement of the theorem is complete.

It remains to see that there is no redundancy in the table of the theorem. Within each family there is certainly no redundancy, as one can see from the formula for $B$ in the theorem. Note that from \cite{AC}, for a lattice equable kite  with vertices $O(0,0),A,B,C$ and $a=OA$ and $b=AB$, one has $\gcd(a,b)=3,4,5$ according to whether $k=9,8,5$ respectively, where $k$ is as defined above. Recall that the members of K1 -- K4 have $k=5,5,8,9$ respectively. It follows that if a kite is congruent to both a member of K$x$ and a member of K$y$, for $x<y$, then the only possibility is that $x=1,y=2$. Suppose we have a kite $Q_1$ in K1, with $n_1^2-5i_1^2=4$, that is congruent to a kite $Q_2$  in K2, with $n_2^2-5i_2^2=1$. From the table at the beginning of the proof of the theorem, we have from the formulas for the areas that $n_1=2n_2$. It follows that, as $n_1^2-5i_1^2=4$ and $n_2^2-5i_2^2=1$, we have $i_1=2i_2$. But then for $Q_1$, we have $b=\frac52(n_1+i_1)=5(n_2+i_2)$, while for $Q_2$, we have $b=5(n_2+2i_2)$. But then $i_2=0$ and hence $n_2=1$, which is excluded in the statement of the theorem.
This completes the proof of Theorem~\ref{Th1}.\hfill\qed

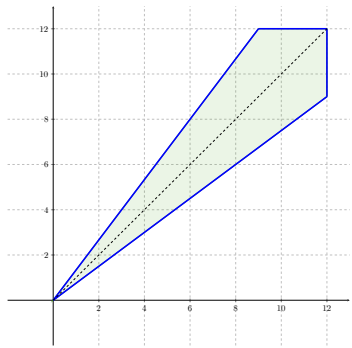
\begin{figure}
\begin{tikzpicture}[scale=.3][line cap=round,line join=round,>=triangle 45,x=1cm,y=1cm]
\begin{axis}[
x=1cm,y=1cm,
axis lines=middle,
grid style=dashed,
ymajorgrids=true,
xmajorgrids=true,
xmin=-2,
xmax=13,
ymin=-2,
ymax=13,
xtick={0,2,...,12},
ytick={0,2,...,12},]
\draw[color=ttzzqq,fill=ttzzqq,fill opacity=0.1] (0,0) -- (12,9) -- (12,12) -- (9,12) -- cycle; 
\draw [line width=2pt,color=qqqqff] (12,12)-- (12,9);
\draw [line width=2pt,color=qqqqff] (12,9)-- (0,0);
\draw [line width=2pt,color=qqqqff] (0,0)-- (9,12);
\draw [line width=2pt,color=qqqqff] (9,12)-- (12,12);
\draw [dashed,line width=.5pt] (0,0)-- (12,12);
\end{axis}
\end{tikzpicture}
\caption{Lattice equable kite with side lengths 3 and 15}\label{F:315}
\end{figure}

\bigskip
\noindent\textbf{Proof of Corollary \ref{co}.}
The convex kites itemized in the corollary are the first members of the families K1, K3 and K4 respectively. We claim all the other kites exhibited in Theorem 1 are darts. To see this, it suffices to note that in all these cases, 
the vector $\overrightarrow{MB}$  lies in the third quadrant. 
For the family K1, when $n>2$, one has $i>0$ and so
 $n^2=5i^2+4<25i^2$, and hence $n-5i<0$. Hence, from the formulas for $B,M$ in the theorem,  $\overrightarrow{MB}=\frac12(n-5i)(2,1)$ is in the third quadrant.

Similarly, for the other three cases, it is enough to note that:
\begin{enumerate}

\item For the family K2, when $n>1$, one has $i>0$ and so
$4n^2=20i^2+4<25i^2$ and hence  $2n-5i<0$.

\item For the family K3, when $n>1$, one has $i>0$ and so
$n^2=2i^2+1<4i^2$ and hence  $n-2i<0$.

\item For the family K4, when $n>1$, one has $i>2$ and so
 $16n^2=8i^2+8<9i^2$ and hence  $4n-3i<0$. 
\end{enumerate}
\hfill\qed


\section{Discussion of  the families K1 -- K4}\label{S:rm}

We now describe the families K1 -- K4 in greater detail. Each case is determined by a very common Pell or Pell-like equation.

For K1, the defining condition is the Pell-like equation $n^2-5i^2=4$.  
It is well known that its solutions, when denoted $(n_{j},i_{j})$, satisfy the following recurrence relation:
\[
(n_{j},i_{j})= 3 (n_{j-1},i_{j-1})-(n_{j-2},i_{j-2}).
\]
 The first few values of $n$ are: 2,3,7,18,47,123,... This is OEIS entry A005248 \cite{OEIS}. 
Here are the first 4 solutions:
\begin{center}
\begin{tabular}{c|c|c|c}
  \hline
   $n$ &  $i$ & $A$ & $B$  \\\hline
 2 &0 &(4,-3) &(4,2)\\
3 &1 &(10,0) &(6,3)\\
7&3 &(24,7) &(14,7)\\
18 &8 &(60,25) &(36,18)\\
\hline
\end{tabular}
\end{center}
The kites with $n=3,7,18$ are shown in Figure \ref{F:K1}.

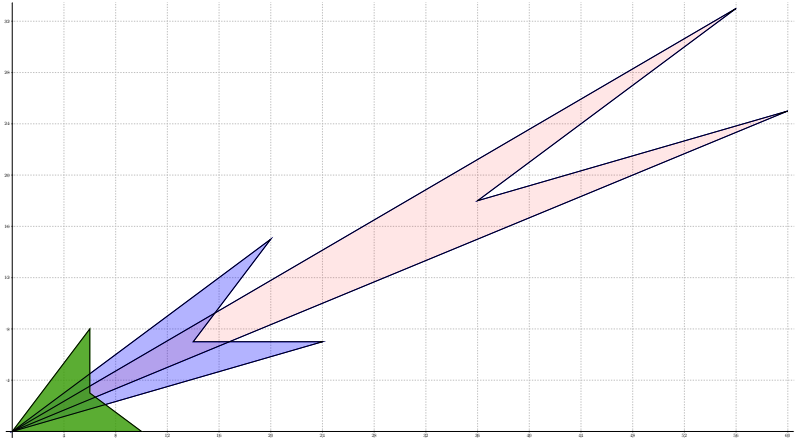
\begin{figure}
\begin{tikzpicture}[scale=.17][line cap=round,line join=round,>=triangle 45,x=1cm,y=1cm]
\begin{axis}[
x=1cm,y=1cm,
axis lines=middle,
grid style=dashed,
ymajorgrids=true,
xmajorgrids=true,
xmin=-0.5,
xmax=60.5,
ymin=-0.5,
ymax=33.5,
xtick={0,4,...,60},
ytick={0,4,...,32},]
\draw[color=qqqqff,fill=red,fill opacity=0.1] (0,0) -- (56,33) -- (36,18) -- (60,25) -- cycle; 
\draw[color=qqqqff,fill=qqqqff,fill opacity=0.3] (0,0) -- (20,15) -- (14,7) -- (24,7) -- cycle; 
\draw[color=ttzzqq,fill=ttzzqq,fill opacity=.8] (0,0) -- (6,8) -- (6,3) -- (10,0) -- cycle; 
\draw[line width=2pt] (0,0) -- (56,33) -- (36,18) -- (60,25) -- cycle; 
\draw[line width=2pt] (0,0) -- (20,15) -- (14,7) -- (24,7)-- cycle; 
\draw[line width=2pt] (0,0) -- (6,8) -- (6,3) -- (10,0)-- cycle; 


\end{axis}
\end{tikzpicture}
\caption{K1 lattice equable kites with $n=3,7,18$}\label{F:K1}
\end{figure}

For K2, the defining condition is the Pell equation $n^2-5i^2=1$.
Its solutions, when denoted $(n_{j},i_{j})$, satisfy the following recurrence relation:
\[
(n_{j},i_{j})= 18 (n_{j-1},i_{j-1})-(n_{j-2},i_{j-2}).
\]
The first few values of $n$ are:  1, 9, 161, 2889, 51841,... This is OEIS entry  A023039  \cite{OEIS}.  
Here are the first 3 cases:
\begin{center}
\begin{tabular}{c|c|c|c}
  \hline
   $n$ &  $i$ & $A$ & $B$  \\\hline
9 &4&(77, 36)&(72,36)\\
161 &72&(1365, 680)&(1288,644)\\
2889 &1292&(24477, 12236)&(23112,11556)\\
\hline
\end{tabular}
\end{center}

For K3, the defining condition is the Pell equation $n^2-2i^2=1$.
Its solutions, when denoted $(n_{j},i_{j})$, satisfy the following recurrence relation:
\[
(n_{j},i_{j})= 6 (n_{j-1},i_{j-1})-(n_{j-2},i_{j-2}).
\]
The first few values of $n$ are:  1, 3, 17, 99, 577,... This is OEIS entry A001541   \cite{OEIS}.   
Here are the first 3 cases:
\begin{center}
\begin{tabular}{c|c|c|c}
  \hline
   $n$ &  $i$ & $A$ & $B$  \\\hline
1 &0&(4,0)&(4,4)\\
3 &2&(16, 12)&(12,12)\\
17 &12&(84, 80)&(68,68)\\
\hline
\end{tabular}
\end{center}

For K4, the defining condition is the Pell-like equation $2n^2-i^2=1$.
Its solutions, when denoted $(n_{j},i_{j})$, satisfy the following recurrence relation:
\[
(n_{j},i_{j})= 6 (n_{j-1},i_{j-1})-(n_{j-2},i_{j-2}).
\]
The first few values of $n$ are: 1, 5, 29, 169, 985,... This is OEIS entry A001653 \cite{OEIS}.    
Here are the first 3 cases:
\begin{center}
\begin{tabular}{c|c|c|c}
  \hline
   $n$ &  $i$ & $A$ & $B$  \\\hline
1 &1&(12, 9)&(12,12)\\
5 &7&(63, 60)&(60,60)\\
29 &41&(360, 357)&(348,348)\\
\hline
\end{tabular}
\end{center}


\section{Proof of Theorem \ref{T:traps}} \label{S:prooftraps}

We follow a strategy suggested in \cite{do}. Consider a trapezoid  with vertices $OABC$ in counterclockwise order, such that $OA$ is the longest side, having an acute angle at $0$; see Figure~\ref{F:traps}. We will not assume that $OABC$ is a LEQ, but we do assume that it is equable with sides of integer length.

\begin{figure}[H]
\begin{tikzpicture}[scale=.4][line cap=round,line join=round,>=triangle 45,x=1cm,y=1cm]
\draw[color=ttzzqq,fill=ttzzqq,fill opacity=0.1] (0,0) -- (12,6) -- (12,8) -- (2,3) -- cycle; 
\draw [line width=1pt,color=qqqqff] (0,0)-- (12,6);
\draw [line width=1pt,color=qqqqff] (12,8)-- (12,6);
\draw [line width=1pt,color=qqqqff] (12,8)-- (2,3);
\draw [line width=1pt,color=qqqqff] (0,0)-- (2,3);
\draw [color=black] (-0.3,-0.3) node {$O$};
\draw [color=black] (12,5.7) node {$A$};
\draw [color=black] (2,.6) node {$A'$};
\draw [color=black] (12.3,8.4) node {$B$};
\draw [color=black] (1.7,3.3) node {$C$};
\draw [line width=1pt,dashed,color=black] (2,3)-- (2,1);
\end{tikzpicture}
\hskip1cm
\begin{tikzpicture}[scale=.4][line cap=round,line join=round,>=triangle 45,x=1cm,y=1cm]
\draw[color=ttzzqq,fill=ttzzqq,fill opacity=0.1] (0,0) -- (12,6) -- (8,6) -- (2,3) -- cycle; 
\draw [line width=1pt,color=qqqqff] (0,0)-- (12,6);
\draw [line width=1pt,color=qqqqff] (8,6)-- (12,6);
\draw [line width=1pt,color=qqqqff] (8,6)-- (2,3);
\draw [line width=1pt,color=qqqqff] (0,0)-- (2,3);
\draw [color=black] (-0.3,-0.3) node {$O$};
\draw [color=black] (12,5.7) node {$A$};
\draw [color=black] (6.3,2.6) node {$A'$};
\draw [color=black] (8.3,6.5) node {$B$};
\draw [color=black] (1.7,3.3) node {$C$};
\draw [line width=1pt,dashed,color=black] (2,3)-- (6,3);
\end{tikzpicture}
\caption{Trapezoids; obtuse and acute}\label{F:traps}
\end{figure}
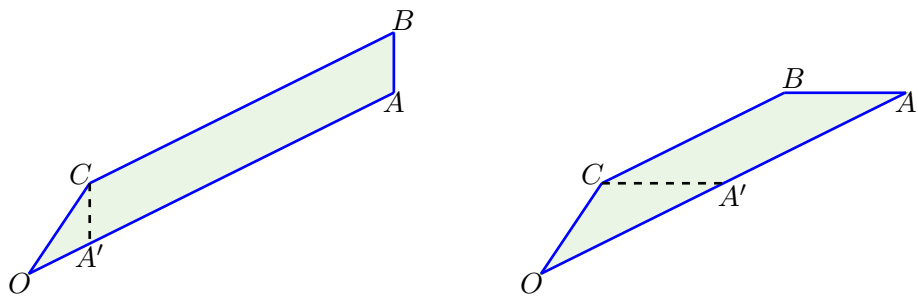

Let $A'$ be the point of the side $OA$ such that $A'C$ is parallel to $AB$. Let $a,c$ denote the lengths of sides $OA,BC$ respectively, and let $h$ be the distance between the parallel side $OA,BC$.  Note that 
$ K=h(a+c)/2$, so $h$ is rational. The triangle $OA'C$ has integer sides. Its area is $h(a-c)/2$, so it is rational.  Hence its area is an integer (so the triangle is Heronian) by the following well known lemma, whose proof we include for completeness (see \cite{Bl}).

\begin{lemma}\label{L:heron}
If a triangle has integer sides and rational area, then its area is an integer.
\end{lemma}

\begin{proof} Consider a triangle with integer sides $x,y,z$ and rational area $T$. By Heron's formula, 
\[
4T=\sqrt{(x+y+z)(-x+y+z)(x-y+z)(x+y-z)}.
\]
Since $(x+y+z)(-x+y+z)(x-y+z)(x+y-z)$ is an integer, and its square root, $4T$, is rational, $4T$ is necessarily an integer. 
Notice that the terms $(x+y+z),(-x+y+z),(x-y+z),(x+y-z)$ have the same parity. So to see that $T$ is an integer, it remains to show that the perimeter $S:=x+y+z$ is even. We have
\[
(4T)^2=S(S-2x)(S-2y)(S-2z)\equiv S^4-2S^3(x+y+z)\equiv -S^4 \pmod4.
\]
But this is possible only if $4T$ and $S$ are even.\hfill\qed
\end{proof}

\smallskip
Since $OABC$ is equable its area is greater than the sum of the lengths of the parallel sides; that is, $h(a+c)/2>a+c$. So $h>2$. 
Now 
the area $ K$ is  $K(OA'C)+hc$ and the perimeter $P$ is $P(OA'C)+2c$.
Hence, as $OABC$ is equable,  
\[
 P(OA'C) = K(OA'C)+(h-2)c  >  K(OA'C).
\]
In the language of \cite{BS}, $OA'C$ is a \emph{perimeter-dominant} Heronian triangle.  
Heronian triangles with this property have been classified \cite{do,BS}.  The complete list is given in Tables \ref{F:pdh1} and \ref{F:pdh2}. We will see that the  triangles of Table~\ref{F:pdh1} arise as the triangle $OA'C$ of an equable trapezoid $OABC$, while the  triangles of Table \ref{F:pdh2} do not.
For an equable trapezoid $OABC$, as above,  
let $f$ denote the length $OA'$. Then from above
$ P(OA'C) = K(OA'C)+(h-2)c=  K(OA'C)+(\frac{2  K(OA'C)}f-2)c $. Rearranging, we have
\begin{equation}\label{E:b}
c=\frac{f\, (Perimeter-Area)}{2 (Area-f)}.
\end{equation}
The constraint that eliminates triangles is the requirement that $c$ is an integer. Table \ref{F:pdh1} gives three special perimeter-dominant Heronian triangles. The table gives the side lengths, perimeter and area, and in the 4th column, we have the values of $c$, given by \eqref{E:b}, corresponding respectively to the three values of $f$ given in the first column.

\begin{table}[h]
\begin{tabular}{c|c|c|c}
  \hline
   Side Lengths $f$&  Perimeter&Area &$c$ \\\hline
(3,4,5) &  12&6& (3,6,15) \\
(5,5,6) & 16&12& $(\frac{10}7,\frac{10}7,2)$ \\
(5,5,8) & 18&12& $(\frac{15}7,\frac{15}7,6) $\\
\hline
\end{tabular}
\smallskip
\caption{Three special perimeter-dominant Heronian triangles}\label{F:pdh1}
\end{table}

{\bf For the \pmb{$3,4,5$} triangle}, given in the first row of Table \ref{F:pdh1}, each choice of side for $OA'$ results in an equable trapezoid. 

\begin{itemize}[leftmargin=*]
\item If $f=3$, \eqref{E:b} gives $c=3$, so $OABC$ has parallel sides of length $c=3$ and $a=c+f=6$. The other sides have length 4 and 5. This trapezoid is the lattice equable right trapezoid shown on the left of Figure \ref{F:lert}.
\item If $f=4$, we obtain $c=6$, so $OABC$ has parallel sides of length $c=6$ and $a=c+f=10$. The other sides have length 3 and 5. This trapezoid is the lattice equable right trapezoid shown on the right of Figure \ref{F:lert}.
\item If $f=5$, we obtain $c=15$, so $OABC$ has parallel sides of length $c=15$ and $a=c+f=20$. The other sides have length 3 and 4. This trapezoid is the lattice equable trapezoid shown in Figure \ref{F:friend}.
\end{itemize}

{\bf For the \pmb{$5,5,6$} triangle}, only the choice of $f=6$ produces an integer value of $c$. Here \eqref{E:b} gives $c=2$, so $OABC$ has parallel sides of length $c=2$ and $a=c+f=8$. The other sides each have length 5. This trapezoid is the lattice equable isosceles trapezoid shown on the left of Figure~\ref{F:it510}.

Similarly, for the \pmb{$5,5,8$} {\bf triangle}, only the choice of $f=8$ produces an integer value of $c$. Here \eqref{E:b} gives $c=6$, so $OABC$ has parallel sides of length $c=6$ and $a=c+f=14$. The other sides each have length 5. This trapezoid is the lattice equable isosceles trapezoid shown on the right of Figure \ref{F:it510}.

\begin{table}[h]
\begin{tabular}{c|c|c|c}
  \hline
   Side Lengths  $f$ & Restrictions& Perimeter&Area  \\\hline
$(3,3x^2+1,3x^2+2)$ & $x\ge 7,\ x^2+1=2y^2$& $6x^2+6$& $6xy$ \\
$(3,3x^2-2,3x^2-1)$  & $x\ge 3,\ x^2-1=2y^2$& $6x^2$& $6xy$  \\
$(4,2x^2+1,2x^2+3)$   & $x\ge 5,\ x^2+2=3y^2$& $4x^2+8$& $6xy$ \\
$(4,4x^2-3,4x^2-1)$    & $x\ge 2,\ x^2-1=3y^2$& $8x^2$& $12xy$  \\
\hline
\end{tabular}
\smallskip
\caption{Infinite families of perimeter-dominant Heronian triangles}\label{F:pdh2}
\end{table}

It remains for us to show that none of the perimeter-dominant Heronian triangles of Table \ref{F:pdh2} give rise to an equable trapezoid.
In Table \ref{F:pdh2}, the variable $x$ is a positive integer subject to a certain Pell, or Pell-like, equation, as given in the Restrictions column of the table.
We examine in turn the four cases given by the four rows of Table \ref{F:pdh2}, and in each case we consider the three given possible values of the side lengths $f$.
In each case, we assume that the value $c$ given by \eqref{E:b} is an integer, and we derive a contradiction.
We will on occasion employ the following elementary result.

\begin{lemma}\label{L:bd}
Consider positive real constants $\alpha,\beta,\gamma,\delta$ and  the real function $g$ defined here below on the open interval $(\delta/\gamma,+\infty)$. Then 
\[
g(r)=\frac{\alpha  r+\beta}{\gamma r-\delta}<1,\qquad \text{ for } \quad  r> \frac{ \beta+\delta}{\gamma-\alpha }.
\]
\end{lemma}


\smallskip
{ \bf 1st Row}:
Here $x\ge 7$ and $x^2+1=2y^2$. In particular, $x$ is odd.
The perimeter is 
$6x^2+6$, and the area is $6xy$. So 
by \eqref{E:b},
\[
c=\frac{f(6x^2+6-6xy)}{2(6xy-f)}=\frac{f(3x^2+3-3xy)}{6xy-f}.
\]
\begin{itemize}[leftmargin=*] 
\item If  $f=3$, we have  $c=\frac{3x^2+3-3xy}{2xy-1}$. Using $y>\frac{x}{\sqrt2}$, Lemma \ref{L:bd} gives
\[
c=\frac{3x^2+3-3xy}{2xy-1}<\frac{(3-\frac{3}{\sqrt2})x^2+3}{\sqrt2x^2-1}<1,
\]
for $x^2> \frac{3+1}{\sqrt2-(3-\frac{3}{\sqrt2})}=\frac{4\sqrt2}{5-3\sqrt2}\cong 7.47$. Thus $c<1$ for $x\ge 3$, contradicting the assumption that $c$ is a positive integer.

\item If $f=3x^2+1$, the value $f$ is even since $x$ is odd. Furthermore,
$f$ is  not divisible by $3$ or any divisor of $x$. Let $z$ be an odd prime common divisor of $f$ and $6xy-f$. So $z$ divides $6xy$ and hence $z$ divides $y$. But $x^2+1=2y^2$, so $z$ divides  $x^2+1$ and $f=3x^2+1$, which is impossible as $z$ does not divide  $x$. Hence $\gcd(f,6xy-f)=2$. Thus, as $c$ is an integer, $3xy-f/2$ divides $3x^2+3-3xy$, say
$2(3x^2+3-3xy)=k(6xy-f)$,
for some positive integer $k$. Using $y>\frac{x}{\sqrt2}$, Lemma \ref{L:bd} gives
\[
k=\frac{2(3x^2+3-3xy)}{6xy-3x^2-1}<\frac{2((3-\frac{3}{\sqrt2})x^2+3)}{(\frac{6}{\sqrt2}-3)x^2-1}<2 ,
\]
for $x^2> \frac{3+1}{(\frac{6}{\sqrt2}-3)-(3-\frac{3}{\sqrt2})}= \frac{4\sqrt2}{9-6\sqrt2 }\cong 11.0$. Thus $k<2$ for  
$x\ge 4$. So the only possibility is $k=1$, but then 
\[
2(3x^2+3-3xy)=6xy-3x^2-1\,\Leftrightarrow\, 9x^2-12xy+7=0,
\]
 which is impossible mod 3.

\item If $f=3x^2+2$, the value $f$ is odd since $x$ is odd. Furthermore, $f$ is  not divisible by $3$ or any divisor of $x$. Let $z$ be a prime common divisor of $f$ and $6xy-f$.  So $z$ divides $y$. But $x^2+1=2y^2$, so $z$ divides  $2x^2+2$ and $f=3x^2+2$, which  is again impossible as $z$ does not divide  $x$. Hence $\gcd(f,6xy-f)=1$. Thus, as $c$ is an integer, $6xy-f$ divides $3x^2+3-3xy$, say
$3x^2+3-3xy=k(6xy-f)$,
for some positive integer $k$. Using $y>\frac{x}{\sqrt2}$, Lemma \ref{L:bd} gives
\[
k=\frac{3x^2+3-3xy}{6xy-3x^2-2}<\frac{(3-\frac{3}{\sqrt2})x^2+3}{(\frac{6}{\sqrt2}-3)x^2-2}<1,
\]
for $x^2> \frac{3+2}{(\frac{6}{\sqrt2}-3)-(3-\frac{3}{\sqrt2})}= \frac{5\sqrt2}{9-6\sqrt2 }\cong 13.7$. Thus $k<1$ for  
$x\ge 4$. So this case is impossible.
\end{itemize}

\smallskip
{ \bf 2nd Row}: 
Here $x\ge 3$ and $x^2-1=2y^2$. In particular, $x$ is odd. Moreover,  $x$ is not a multiple of 3 if and only if $y$ is  a multiple of 3.  
The perimeter is 
$6x^2$, and the area is $6xy$. So 
by \eqref{E:b},
\[
c=\frac{f(3x^2-3xy)}{6xy-f}=\frac{3fx(x-y)}{6xy-f}.
\]

\begin{itemize}[leftmargin=*] 
\item If $f=3$, we have  $c= \frac{3x(x-y)}{2xy-1}$. 
As $x$ is not a multiple of 3 if and only if $y$ is  a multiple of 3, so $2xy-1$ is not divisible by 3.  Furthermore, $2xy-1$ is not divisible by  prime divisor of $x$.
So as $2xy-1$ divides $3x(x-y)$, necessarily $2xy-1$ divides $x-y$, say
$x-y=k(2xy-1)$,
for some positive integer $k$. So $x(2ky-1) =k-y$ and thus $x =\frac{k-y}{2ky-1}\le \frac{k-1}{2k-1}\le 1$, contradicting the assumption that $x\ge 3$.

\item If $f=3x^2-2$, the value $f$ is odd, and not divisible by $3$ or any divisor of $x$. Let $z$ be a prime common divisor of $f$ and $6xy-f$. 
 So $z$ divides $y$. But $x^2-1=2y^2$, so $z$ divides  $x^2-1$ and $f=3x^2-2$, which is impossible as $z$ does not divide  $x$. Hence $\gcd(f,6xy-f)=1$. Thus, as $c$ is an integer, $6xy-f$ divides $3x^2-3xy$. Since $f$ is odd and not divisible by $3$ or any divisor of $x$, so $6xy-f$ divides $x-y$, say
$x-y=k(6xy-3x^2+2)$,
for some positive integer $k$. Now $2y>x$ (as $4y^2>2y^2+y^2 = x^2-1+y^2\ge x^2$), so
\[
k=\frac{x-y}{6xy-3x^2+2}<\frac{x-y}{6xy-3x^2}=\frac{x-y}{3x(2y-x)}\le \frac{x-y}{3x}< 1,
\]
which is impossible.

\item If $f=3x^2-1$, the value $f$ is even (as $x$ is odd), and not divisible by $3$ or any divisor of $x$. Moreover, observing the previous equality in $\Z_4$ implies  $f$ is exactly divisible by 2. Let $z$ be an odd prime common divisor of $f$ and $6xy-f$. 
So $z$ divides $y$. But $x^2-1=2y^2$, so $z$ divides  $x^2-1$ and $f=3x^2-1$, which is impossible as $z$ does not divide  $x$. Hence $\gcd(f,6xy-f)=2$. Thus, as $c$ is an integer, $3xy-f/2$ divides $3x^2-3xy$. Since $f$ is  not divisible by $3$ or any divisor of $x$, neither is $3xy-f/2$, so $3xy-f/2$ divides $x-y$, say
$2(x-y)=k(6xy-3x^2+1)$,
for some positive integer $k$. Now $2y>x$ so
\[
\frac{2(x-y)}{6xy-3x^2+1}<\frac{2(x-y)}{6xy-3x^2}=\frac{2(x-y)}{3x(2y-x)}\le \frac{2(x-y)}{3x}< 1,
\]
which is impossible.
\end{itemize}

\smallskip
{ \bf 3rd Row}: 
Here $x\ge 5$ and $x^2+2=3y^2$, which observed in $\Z_4$ gives $x$ and $y$ are both odd, and $x$ is not divisible by 3.
 The perimeter is 
$4x^2+8$, and the area is $6xy$. So 
by \eqref{E:b},
\[
c=\frac{f(2x^2+4-3xy)}{6xy-f}.
\]

\begin{itemize}[leftmargin=*]
\item $f=4$, we have  $c=\frac{2(2x^2+4-3xy)}{3xy-2}$.
As $x,y$ are both odd, $3xy-2$ is odd, and thus $3xy-2$ divides $2x^2+4-3xy$, say
$2x^2+4-3xy=k(3xy-2)$,
for some positive integer $k$. Using $y> x/\sqrt{3}$, Lemma \ref{L:bd} gives
\[
\frac{2x^2+4-3xy}{3xy-2}<\frac{(2-\sqrt{3})x^2+4}{\sqrt{3}x^2-2}<1,
\]
for $x^2> \frac{4+2}{\sqrt{3}-(2-\sqrt{3})}=\frac{6}{2\sqrt3-2}\cong 4.1$. Thus $k<1$ for $x\ge 3$, contradicting the assumption that $k$ is a positive integer.

\item If $f=2x^2+1$, the value $f$  is odd, divisible by $3$ but not by any divisor of $x$.
Let $z>3$ be a prime common divisor of $f$ and $6xy-f$. So $z$ divides  $y$. But $x^2+2=3y^2$, so $z$ divides  $x^2+2$ and $f=2x^2+1$, which is impossible as $z$ does not divide  $x$. Hence $\gcd(f,6xy-f)=3$. Thus, as $c$ is an integer, $2xy-f/3$ divides $2x^2+4-3xy$, say
$3(2x^2+4-3xy)=k(6xy-2x^2-1)$,
for some positive integer $k$. Using $y> x/\sqrt{3}$, Lemma \ref{L:bd} gives
\[
\frac{3(2x^2+4-3xy)}{6xy-2x^2-1}<\frac{3((2-\sqrt3)x^2+4)}{(2\sqrt3-2)x^2-1}<1
\]
for $x^2> \frac{12+1}{(2\sqrt3-2)-3(2-\sqrt3)}=\frac{13}{5\sqrt3-8}\cong 19.69$. Thus $k<1$ for $x\ge 5$, contradicting the assumption that $k$ is a positive integer.

\item If $f=2x^2+3$, the value $f$ is odd, and is not divisible by $3$ ($x$ is not divisible by 3) or any divisor of $x$.
Let $z$ be a prime common divisor of $f$ and $6xy-f$. So $z$ divides $y$. But $x^2+2=3y^2$, so $z$ divides  $x^2+2$ and $f=2x^2+3$, which is impossible as $z$ does not divide  $1$. Hence $\gcd(f,6xy-f)=1$. Thus $2xy-f/3$ divides $2x^2+4-3xy$, say
$2x^2+4-3xy=k(6xy-2x^2-3)$,
for some positive integer $k$. Using $y> x/\sqrt{3}$, Lemma \ref{L:bd} gives
\[
\frac{2x^2+4-3xy}{6xy-2x^2-3}<\frac{(2-\sqrt3)x^2+4}{(2\sqrt3-2)x^2-3}<1,
\]
for $x^2> \frac{4+3}{(2\sqrt3-2)-(2-\sqrt3)}=\frac{7}{3\sqrt3-4}\cong 5.85$. Thus $k<1$ for $x\ge 3$, contradicting the assumption that $k$ is a positive integer.
\end{itemize}

\smallskip
{ \bf 4th Row}: 
Here $x\ge 2$ and $x^2-1=3y^2$. Note $x$ is not divisible by 3. Furthermore, $x,y$ have opposite parity, so $xy$ is even. 
 The perimeter is 
$8x^2$, and the area is $12xy$. So 
by \eqref{E:b},
\[
c=\frac{f(4x^2-6xy)}{12xy-f}=\frac{2fx(2x-3y)}{12xy-f}.
\]

\begin{itemize}[leftmargin=*]
\item If $f=4$, we have  $c=\frac{2x(2x-3y)}{3xy-1}$. 
Now $3xy-1$ is odd and not divisible by any prime divisor of $x$. So $3xy-1$ divides $2x-3y$, say
$2x-3y=k(3xy-1)$,
for some positive integer $k$. Thus $x=\frac{k-3y}{3ky-2}<  \frac{k}{3k-2}\le 1$, contradicting the assumption that $x\ge 2$.

\item If $f=4x^2-3$, the value
$f$ is odd and not divisible by 3 (as $x$ is not divisible by 3) or any of the prime divisors of $x$. Let $z$ be a prime common divisor of $f$ and $12xy-f$. So $z$ divides $y$.
But $x^2-1=3y^2$, so $z$ divides  $x^2-1$ and $f=4x^2-3$, which is impossible as $z$ does not divide  $x$. Hence $\gcd(f,12xy-f)=1$. Thus,  as $c$ is an integer and $12xy-f$ is odd and $\gcd(x,12xy-f)=1$, we have that $12xy-f$ divides $2x-3y$, say
$2x-3y=k(12xy-4x^2+3)$,
for some positive integer $k$. 
Now $2y\ge x$ (as $4y^2=3y^2+y^2=x^2-1+y^2\ge x^2$), so
\[
k=\frac{2x-3}{12xy-4x^2+3}<\frac{2x}{2x^2+3}<1,
\]
for $x\ge 1$. So this is impossible.

\item If $f=4x^2-1$, the value $f$ is odd. It is divisible by 3, but not by any of the prime divisors of $x$. Let $z>3$ be a prime common divisor of $f$ and $12xy-f$. So $z$ divides $y$.
But $x^2-1=3y^2$, so $z$ divides  $x^2-1$ and $f=4x^2-1$, which is impossible as $z$ does not divide  $x$. Hence $\gcd(f,12xy-f)=3$. Thus, as $c$ is an integer and $12xy-f$ is odd and $\gcd(x,12xy-f)=1$, $(12xy-f)/3$ divides $2x-3y$, say
$3(2x-3y)=k(12xy-4x^2+1)$,
for some positive integer $k$. 
Using $2y\ge x$ again, 
\[
k=\frac{3(2x-3y)}{12xy-4x^2+1}<\frac{3x}{4x^2+2}<1,
\]
for $x\ge 1$. So this case is also impossible.
\end{itemize}

This completes the proof of Theorem \ref{T:traps}.


\section{Cyclic quadrilaterals; Proof of Theorem \ref{Th4}}\label{S:LECQs}

Consider an  equable cyclic quadrilateral with integer sides $a\le b\le c \le d$. By Brahmagupta's formula's,
\begin{equation}\label{E:bg}
(-a+b+c+d)(a-b+c+d)(a+b-c+d)(a+b+c-d)=16(a+b+c+d)^2.
\end{equation}
Notice each sum above is of same parity and hence even.
Let $2z=-a+b+c+d, 2y=a-b+c+d, 2x=a+b-c+d, 2w=a+b+c-d$ and note by elementary geometry  $0<w\le x\le y\le z$. Hence  
\begin{equation}\label{E:bg2}
wxyz=(w+x+y+z)^2.
\end{equation}
Conversely,
\begin{align}\label{E:abcd}
\begin{split}
a&=\frac12(-w+x+y+z), \quad c=\frac12(w-x+y+z),\\
c&=\frac12(w+x-y+z), \quad\ \ d=\frac12(w+x+y-z).
\end{split}
\end{align}

Since $w\le x\le y\le z$, one has $w+x+y-z=2d>0$, and so
\begin{equation}\label{E:z3y}
z< w+x+y.
\end{equation}
From \eqref{E:bg2},
\[
z=\frac{wxy-2(w+x+y)\pm \sqrt{w^2x^2y^2-4wxy(w+x+y)}}{2}.
\]
In particular, $w^2x^2y^2-4wxy(w+x+y)\ge 0$, so
\begin{equation}\label{E:pos}
wxy\ge 4(w+x+y).
\end{equation}
Suppose for the moment that $z$ is given by the positive square root. Then as $z\le w+x+y$, we would have
\[
z=\frac{wxy-2(w+x+y)+ \sqrt{w^2x^2y^2-4wxy(w+x+y)}}{2} \le w+x+y,
\]
so 
$\sqrt{w^2x^2y^2-4wxy(w+x+y)} \le 4(w+x+y) -wxy \le 0$, by \eqref{E:pos}. So the positive square root only gives a solution when the square root is zero,  and this is encompassed by the case where $z$ is given by the negative square root.
 So we may assume that
\begin{equation}\label{E:z}
z=\frac{wxy-2(w+x+y)- \sqrt{w^2x^2y^2-4wxy(w+x+y)}}{2}.
\end{equation}
Then $y\le z$ gives
$\sqrt{w^2x^2y^2-4wxy(w+x+y)} \le wxy-2(w+x+2y)$.
Squaring 
and simplifying gives
\begin{equation}\label{E:ine}
wxy^2\le(w+x+2y)^2.
\end{equation}
Consequently, as $w\le x\le y$, we have $wxy^2\le  (4y)^2$, so $wx\le16$.
Hence $w^2\le wx\le 16$. In particular $w\le 4$.  Rewriting (\ref{E:ine}) and using the previous we get 
$wx>4$ and 
\[
y\le \frac{w+x}{\sqrt{wx}-2}\le \frac{4+16}{\sqrt{5}-2}< 85.
\] 
There are 63 triples $(w,x,y)$ satisfying $w\le x\le y\le 84$ with  $5\le wx\le16$.
Testing each possibility by replacing it in (\ref{E:z}), and checking if $z$ is an integer and $y\le z<w+x+y$,  we obtain only four solutions for $(w,x,y,z)$, namely
\[
(1,9,10,10),\ (2,5,5,8),\ (3,3,6,6),\ (4,4,4,4),
\]
which by \eqref{E:abcd}, give the corresponding four solutions announced for $(a,b,c,d)$:
\begin{equation}\label{E:sols}
(14,6,5,5),\ (8,5,5,2),\ (6,6,3,3),\ (4,4,4,4).
\end{equation}
It remains to consider the possible ordering of the sides. For this, we use the hypothesis that the quadrilateral is a LEQ. The numbers $2^2,3^2,4^2,6^2,8^2,14^2$  cannot be written as the sum of two nonzero squares, so the corresponding sides are each either vertical or horizontal. Thus the last two cases are obviously the $3\times 6$ rectangle and the $4\times 4$ square, while it is easy to see that the first two cases are necessarily  isosceles trapezoids, arranged either vertically or horizontally, and hence congruent to those of Figure~\ref{F:it510}. This completes the proof of Theorem \ref{Th4}.

\begin{rem}
To put Theorem \ref{Th4} in context, note that there are cyclic equable quadrilaterals whose sides lengths, in cyclic order, are any of the permutations of the lists in \eqref{E:sols}. For example, the equable  kite with side length 6,3,3,6 is cyclic, but cannot be realized as a LEQ; see Figure~\ref{F:36}.
\end{rem}


\section{Proof of Theorem~\ref{T:irrat}}\label{S:irratdiag}

Let $a,b,c,d$ denote the lengths of $OA,AB,BC,CO$ respectively.
Consider a diagonal that is internal to $OABC$. By relabelling the vertices if necessary, we may take this diagonal to be $OB$. Assume that $p=OB$ is rational. So, since $p$ is the distance between lattice points, $p$ is an integer. Moreover, the area of triangles $OAB$ and $OBC$ are rational, as they are given by 
$\frac12\Vert \overrightarrow{OA}\times\overrightarrow{OB}\Vert$ and $\frac12\Vert \overrightarrow{OB}\times\overrightarrow{OC}\Vert$
respectively.
Hence by Lemma~\ref{L:heron}, both triangles have integer area. So they are Heronian triangles. 

Notice that $p>4$. Indeed, let $h_A,h_C$ denote the respective altitudes of the triangles $OAB,OBC$ from the side $OB$. So $a+b>2h_A$ and $c+d>2h_C$. Thus
\[
a+b+c+d=K=K_A+K_C=\frac12(h_A+h_C)p < \frac{p}4(a+b+c+d).
\]

Now observe that one of triangles $OAB,OBC$ is perimeter dominant, as otherwise one would have
\[
a+b+c+d=K=K_A+K_C\ge (a+b+p)+(c+d+p),
\]
which is impossible. Assume without loss of generality that $OAB$ is perimeter dominant. For the reader's convenience, we recall the possibilities for perimeter dominant  Heronian triangles in Table~\ref{F:pdh3}. We treat each row in order, and consider the possibilities given that $p>4$. Our strategy is to show that in each case, the triangle $BCO$ is also perimeter dominant. For this, we use the following equation,
\begin{equation}\label{E:pd}
\Delta_C :=P_C-K_C=K_A+p-a-b,
\end{equation}
which holds as $P_C=p+c+d$ and $K_C=K-K_A=a+b+c+d-K_A$, since $OABC$ is equable. Then when we consider row $i$ for $OAB$, we need only consider rows $\ge i$ for $BCO$. To avoid confusion, when we consider $BCO$, we relabel the variables in these rows $u,v$ instead of $x,y$.

\begin{table}[H]
\resizebox{\columnwidth}{!}{%
\begin{tabular}{c|c|c|c|c}
  \hline
   Side Lengths & Restrictions& Peri. $P$ &Area $K$& $\Delta=P-K$ \\ \hline
(3,4,5) &-&  12&6 &6\\
(5,5,6) &-& 16&12&4 \\
(5,5,8) &-& 18&12&6\\
$(3,3x^2+1,3x^2+2)$ & $x\ge 7,\ x^2+1=2y^2$& $6x^2+6$& $6xy$ &$6(x^2-xy+1)$ \\
$(3,3x^2-2,3x^2-1)$  & $x\ge 3,\ x^2-1=2y^2$& $6x^2$& $6xy$  &$6(x^2-xy)$\\
$(4,2x^2+1,2x^2+3)$   & $x\ge 5,\ x^2+2=3y^2$& $4x^2+8$& $6xy$& $4x^2-6xy+8$\\
$(4,4x^2-3,4x^2-1)$    & $x\ge 2,\ x^2-1=3y^2$& $8x^2$& $12xy$  & $8x^2-12xy$\\
\hline
\end{tabular}
}
\caption{Perimeter-dominant Heronian triangles}\label{F:pdh3}
\end{table}

{\bf Row 1.} Here $K_A=6$ and $p=5$ and $\{a,b\}=\{3,4\}$. From \eqref{E:pd},
\[
\Delta_C=K_A+p-a-b=4.
\]
So $BCO$ is  perimeter dominant. Consultation of Table~\ref{F:pdh3} shows that the only perimeter dominant  Heronian triangle having a side of length 5 with perimeter minus area equal to 4 is the 5,5,6 triangle. Thus $OABC$ has sides 3,4,5,6, being composed of a 3,4,5 right angled triangle joined along its hypotenuse with a 5,5,6 triangle. It follows that $OABC$ is the right trapezoid LEQ shown on the left of Figure~\ref{F:lert}, which indeed has a diagonal of length 5.
 
 
 {\bf Row 2.} Here $K_A=12$. If $p=5$, then  $\{a,b\}=\{5,6\}$ and from \eqref{E:pd},
\[
\Delta_C=K_A+p-a-b=6.
\]
So $BCO$ is perimeter dominant. From Table~\ref{F:pdh3}, there is only perimeter dominant  Heronian triangle in rows $\ge 2$ having a side of length 5 with perimeter minus area equal to 6, namely the 5,5,8 triangle. Hence $OABC$ has two consecutive sides of length 5,6, two other consecutive sides of length 5,8 and those of length 6 and 8 must be horizontal or vertical, in particular either parallel or perpendicular to one another. Consideration of the possibilities leads one to the 6,8,10 right angle triangle with a vertex in the middle of the hypotenuse; we eliminate this possibility as failing the definition of a quadrilateral.

  If $p=6$, then  $\{a,b\}=\{5,5\}$ and
\[
\Delta_C=K_A+p-a-b=8.
\]
So $BCO$ is perimeter dominant. From Table~\ref{F:pdh3}, there is no perimeter dominant  Heronian triangle having a side of length 6 with perimeter minus area equal to 8. So this case is impossible.


{\bf Row 3.} Here $K_A=12$.  If $p=5$, then  $\{a,b\}=\{5,8\}$ and
\[
\Delta_C=K_A+p-a-b=4.
\]
So $BCO$ is perimeter dominant. But from Table~\ref{F:pdh3}, there is no perimeter dominant  Heronian triangle in rows $\ge 3$ having a side of length 5 with perimeter minus area equal to 4.

If $p=8$, then  $\{a,b\}=\{5,5\}$ and
\[
\Delta_C=K_A+p-a-b=10.
\]
So $BCO$ is perimeter dominant. But from Table~\ref{F:pdh3}, there is no perimeter dominant  Heronian triangle having a side of length 8 with perimeter minus area equal to 10. So this case is also impossible.


{\bf Row 4.} Here $K_A=6xy$, where $x\ge 7$ and $x^2+1=2y^2$. If  $p=3x^2+1$, then $\{a,b\}=\{3,3x^2+2\}$ and 
\[
\Delta_C=K_A+p-a-b= 6xy+3x^2+1-3-3x^2-2=6xy-4.
\]
So $BCO$ is perimeter dominant. We again examine Table~\ref{F:pdh3}. Notice that $P_C-K_C\equiv 2\pmod 6$, so we need only consider the bottom 2 rows.   For $BCO$ in row 6, 
for the perimeter minus area, $4u^2+8-6uv$, to be congruent to $2\pmod 6$, we would require $u$ to be a multiple of 3, which is impossible because $u^2+2=3v^2$. It remains to consider row 7. Suppose  $OBC$ has side lengths $4,4u^2-3,4u^2-1$, with $u\ge 2$ and $ u^2-1=3v^2$. So either $p=3x^2+1=4u^2-3$ or $3x^2+1=4u^2-1$. But in both cases  $x$ must be even, which  is impossible as $x^2+1=2y^2$.

If  $p=3x^2+2$, then $\{a,b\}=\{3,3x^2+1\}$ and
\[
\Delta_C=K_A+p-a-b= 6xy+3x^2+2-3-3x^2-1=6xy-2.
\]
So $BCO$ is perimeter dominant. Notice that now $P_C-K_C\equiv 4\pmod 6$, so we need again only consider rows 6 and 7 of Table~\ref{F:pdh3}. For $BCO$ in row 6, 
for the perimeter minus area, $4u^2+8-6uv$, to be congruent to $4\pmod 6$, we would require $4u^2\equiv 2\pmod 6$, or equivalently, $2u^2\equiv 1\pmod 3$. But this is impossible. For $BCO$ in row 7, for the perimeter minus area, $8u^2-12uv$, to be congruent to $4\pmod 6$, we require $2u^2\equiv 4\pmod 6$, or equivalently, $u^2\equiv 2\pmod 3$. But this is again impossible.


{\bf Row 5.} Here $K_A=6xy$, where $x\ge 3$ and $x^2-1=2y^2$, in particular $x$ is odd.  If  $p=3x^2-2$, then $\{a,b\}=\{3,3x^2-1\}$ and 
\[
\Delta_C=K_A+p-a-b= 6xy+3x^2-2-3-3x^2+1=6xy-4.
\]
Hence $BCO$ is perimeter dominant and  $\Delta_C\equiv 2 \pmod{6}$ so we consider only the two last rows. If $\Delta_C=4u^2-6uv+8\equiv 2 \pmod{6}$, then $u$ is a multiple of 3, which is impossible because $u^2+2=3v^2$. Otherwise, either $3x^2-2=4u^2-3$ or $3x^2-2=4u^2-1$, with both $x^2-1=2y^2$ and $u^2-1=3v^2$. By substituting $x$ and $u$ respectively by $y$ and $v$ we get in the first case $y^2=2v^2$ a contradiction,  and in the second   $3y^2-6v^2=1$ another contradiction. 

If  $p=3x^2-1$, then$\{a,b\}=\{3,3x^2-2\}$ and 
\[
\Delta_C=K_A+p-a-b= 6xy+3x^2-1-3-3x^2+2=6xy-2.
\]
Once again, $BCO$ is perimeter dominant,  $\Delta_C\equiv 4 \pmod{6}$, and $4u^2+8-6uv\equiv 4\pmod{6}$ implies $2\equiv u^2\pmod{3}$ a contradiction. Otherwise, either $3x^2-1=4u^2-3$ or $3x^2-1=4u^2-1$. In the first case, $3x^2=4u^2-2$  is impossible since $x$ is odd. The second  implies $3x^2=4u^2$, another contradiction. 
 

{\bf Row 6}.  Here $K_A=6xy$, where $x\ge 5$ and $x^2+2=3y^2$. Moreover, $x$ is odd by observing  $x^2+2=3y^2$ in $\Z_4$. If  $p=2x^2+1$, then $\{a,b\}=\{4,2x^2+3\}$ and 
\[
\Delta_C=K_A+p-a-b= 6xy+2x^2+1-4-2x^2-3=6xy-6.
\]
So $BCO$ is perimeter dominant. Since $x^2+2=3y^2$, $x$ is not divisible by 3, so $p=2x^2+1$ is divisible by 3. We consider the rows 6 and 7 of  Table~\ref{F:pdh3}.
For $BCO$ in row 6,  $u^2+2=3v^2$ so $u$ is not divisible by 3. Thus $p$ cannot equal the side  $2u^2 +3$. If $p=2u^2+1$, then $p=2x^2+1=2u^2+1$, so $x=u$, and thus $y=v$.
But then, as the perimeter minus area is $6xy-6=4u^2+8-6uv$, we have $ 6uv-2u^2=7$, which is impossible.

 For $BCO$ in row 7, as the perimeter minus area is divisible by 6, we have $8u^2-12uv\equiv0\pmod 6$, so $u$ must be a multiple of 3. But this is impossible as 
$u^2-1=3v^2$. 

If  $p=2x^2+3$, then $\{a,b\}=\{4,2x^2+1\}$ and 
\[
\Delta_C=K_A+p-a-b= 6xy+2x^2+3-4-2x^2-1=6xy-2.
\]
Once again, $BCO$ is perimeter dominant,  $\Delta_C\equiv 4 \pmod{6}$, and $4u^2+8-6uv\equiv 4\pmod{6}$ implies $2\equiv u^2\pmod{3}$ a contradiction. Otherwise, either $2x^2+3=4u^2-3$ or $2x^2+3=4u^2-1$. In the first case, $3 = 2u^2-x^2$ is impossible since $3\nmid ux$. The second  implies $2=2u^2-x^2$, another contradiction as $x$ is odd.


{\bf Row 7}.  Here $K_A=12xy$, where $x\ge 2$ and $x^2-1=3y^2$. 
If   $p=4x^2-3$, then $\{a,b\}=\{4,4x^2-1\}$ and 
\[
\Delta_C=K_A+p-a-b= 12xy+4x^2-3-4-4x^2+1=12xy-6.
\]
So $BCO$ is perimeter dominant. 
For $BCO$ in row 7, as the perimeter minus area is divisible by 6, we have $8u^2\equiv0\pmod 6$, so $u$ must be a multiple of 3. But this is impossible as 
$u^2-1=3v^2$. 

This completes the proof of the theorem.

\bibliographystyle{amsplain}
{}

\bigskip

\bigskip

\bigskip

COLL\`EGE CALVIN

GENEVA

SWITZERLAND 1211

\textit{E-mail address}: \texttt{christian.aebi@edu.ge.ch}

\bigskip

\bigskip

DEPARTMENT OF MATHEMATICS

LA TROBE UNIVERSITY

MELBOURNE

AUSTRALIA 3086

\textit{E-mail address}: \texttt{G.Cairns@latrobe.edu.au}\bigskip


\end{document}